\documentclass[12pt,oneside]{amsart}
\usepackage{amsmath}
\usepackage{amsthm}
\usepackage{amsfonts}
\usepackage{amssymb}
\usepackage{enumerate}
\newtheorem{theorem}{Theorem}

\newtheorem{proposition}[theorem]{Proposition}
\newtheorem{corollary}[theorem]{Corollary}

\theoremstyle{definition}

\newtheorem{example}[theorem]{Example}

\theoremstyle{remark}
\newtheorem{remark}[theorem]{Remark}
\newcommand{\DD}{{\mathbb D}}

\DeclareMathOperator{\Aut}{Aut} 
 
 \DeclareMathOperator{\re}{Re}

\DeclareMathOperator{\im}{Im}

\renewcommand{\phi}{\varphi}


\begin{document}

\title[Convex tube domains]{Regularity of complex geodesics and (non)-Gromov hyperbolicity of convex tube domains}

\address{Carl von Ossietzky Universitat Oldenburg, Institut fur Mathematik, Postfach 2503, ¨
D-26111 Oldenburg, Germany}

\address{Institute of Mathematics, Faculty of Mathematics and Computer Science, Jagiellonian
University,  \L ojasiewicza 6, 30-348 Krak\'ow, Poland}

\author{Peter Pflug}\email{Peter.Pflug@uni-oldenburg.de}
\author{W\l odzimierz Zwonek}\email{wlodzimierz.zwonek@im.uj.edu.pl}
\thanks{The paper was initiated while the second author was at the research stay at the Carl von Ossietzky University of Oldenburg supported by the Alexander von Humboldt Foundation. The second Author was also partially supported by the OPUS grant no. 2015/17/B/ST1/00996 of the National Science Centre, Poland. }


\begin{abstract} We deliver examples of non-Gromov hyperbolic tube domains with convex bases (equipped with the Kobayashi distance). This is shown by providing a criterion on non-Gromov hyperbolicity of (non-smooth) domains.The results show the similarity of geometry of the bases of non-Gromov hyperbolic tube domains with the geometry of non-Gromov hyperbolic convex domains. A connection between the Hilbert metric of a convex domain $\Omega$ in $\mathbb R^n$ with the Kobayashi distance of the tube domain over the domain $\Omega$  is also shown. Moreover, continuity properties up to the boundary of complex geodesics in tube domains with a smooth convex bounded base are also studied in detail.
\end{abstract}
\maketitle

\section{Introduction} Although the results presented below will not be referring only to tube domains these domains are good models for our considerations. The starting point for the research undertaken in the paper are the results and problems discussed in papers \cite{Zim 2016a}, \cite{Zim 2016b} where Gromov hyperbolicity was studied in classes of ($\mathbb C$-)convex domains. A source of examples given there was a class of semitube domains. It is natural to study another class of unbounded domains with many symmetries and this class could be the one of tube domains.

The methods presented below lead in a natural way to phenomena that show similarity between the notion of the non-Gromov hyperbolicity of tube domains with respect to the Kobayashi distance with the same notion for convex domains in $\mathbb R^n$ equipped with the Hilbert metric. This is somehow astonishing and may be another reason for undertaking research towards a better understanding of  the Kobayashi geometry of convex tube domains. For this reason it may also be necessary to understand the boundary behavior of geodesics (complex and real) with respect to the Kobayashi distance in convex tube domains. Recall that in the proofs of positive results on Gromov hyperbolicity A. Zimmer extensively used the notions in (among others) the case of unbounded semitube domains (see \cite{Zim 2016b}). And making use of recent results of Zaj\c ac (\cite{Zaj 2015a} and \cite{Zaj 2015b}) we present detailed general results on the boundary behavior of geodesics in tube domains over bounded smooth, strictly convex bases.

\section{Definitions and results}

\subsection{Kobayashi pseudodistance, Kobayashi hyperbolicity, geo\-de\-sics, Lempert Theorem}
For a domain $D\subset\mathbb C^n$ we define \textit{the Kobayashi pseudodistance $k_D$} as follows. The function $k_D$ is the largest pseudodistance not exceeding \textit{the Lempert function $l_D$}
\begin{equation}
l_D(w,z):=\inf\{p(0,\lambda):\exists \phi\in\mathcal O(\mathbb D,D): \phi(0)=w,\phi(\lambda)=z\},\;w,z\in D.
 \end{equation}
 In the above formula $p$ denotes the Poincar\'e distance on the unit disc $\mathbb D$, $\mathbb D:=\{\lambda\in\mathbb C:|\lambda|<1\}$. We also put $\mathbb T:=\partial\mathbb D$. In the case $k_D$ is actually distance, i. e. $k_D(w,z)>0$, $w,z\in D$, $w\neq z$, we call the domain $D$ \textit{Kobayashi hyperbolic}. If $D$ is a bounded domain then $D$ is always Kobayashi hyperbolic. In the case $D$ is a convex domain then we may easily get that $D$ is linearly isomorphic with $\mathbb C^k\times D^{\prime}$, where $0\leq k\leq n$ and $D^{\prime}\subset\mathbb C^{n-k}$ is a Kobayashi hyperbolic convex domain (see e. g. Proposition 1.2 in \cite{Bra-Sar 2009}).

 The fundamental Lempert Theorem states among others that if $D\subset\mathbb C^n$ is a Kobayashi hyperbolic convex domain then for any distinct points $w,z\in D$ there is a \textit{ complex geodesic} $f:\mathbb D\to D$ passing through $w,z$ which, by definition means that $f$ is holomorphic, $w,z$ lie in the image of $f$, and there is a \textit{left inverse to $f$}, i. e. a holomorphic function $F:D\to\mathbb D$ such that $F\circ f$ is the identity. Composing, if necessary, with an automorphism of $\mathbb D$ we may always assume that $f(0)=w$, $f(s)=z$ for some $s\in(0,1)$. A simple consequence of the existence of complex geodesics is the existence of a \textit{ real geodesic} passing through $w$ and $z$ by which we mean a mapping $\gamma:(-1,1)\to D$ such that for any $-1<s_0<s_1<1$ we have $p(s_0,s_1)=k_{\mathbb D}(s_0,s_1)=k_D(\gamma(s_0),\gamma(s_1))$. If $f:\mathbb D\to D$ is a complex geodesic, then $\gamma:=f_{(-1,1)}:(-1,1)\to D$ is a real geodesic. And in the case the uniqueness of complex geodesics is guaranteed --- for instance for strictly convex bounded domains or (for the need of our paper) for tube domains with bounded strictly convex bounded bases (see Proposition~\ref{proposition:uniqueness})--- any real geodesic $\gamma:(-1,1)\to D$ is of the form $f\circ a_{|(-1,1)}$, where $f:\mathbb D\to D$ is a complex geodesic and $a$ is some automorphism of $\mathbb D$.

 The Lempert Theorem was originally proven by L. Lempert (see e. g. \cite{Lem 1981}). As a good reference for the Lempert theory as well as basic properties of the Kobayashi pseudodistance, Kobayashi hyperbolicity, complex geodesics that we shall use in the paper we refer the Reader to \cite{Jar-Pfl 1993}.

 \subsection{Non-Gromov hyperbolicity}

Let $d_X:X\times X\to[0,\infty)$ be a pseudodistance on a set $X$ (i. e. a function satisfying all the properties of the distance with the exception that $d_X(x,y)$ may be $0$ for $x\neq y$). Define
\begin{multline}
S_X(x,y,z,w):=\\
d_X(x,z)+d_X(y,w)-\max\{d_X(x,y)+d_X(z,w),d_X(y,z)+d_X(x,w)\},\\
x,y,z,w\in X.
\end{multline}
One of possible (equivalent) definitions of the \textit{non-Gromov hyperbolic} space $(X,d_X)$ (compare formula (2.1) in \cite{Kar-Nos 2002}) is that
\begin{equation}
S_X:=\sup\{S_X(x,y,z,w):x,y,z,w\in X\}=\infty.
\end{equation}
Let us mention here that the above definition is given originally for $d_X$ being the distance. We extend its definition in our paper so that the formulation of some results could be simplified.

In our paper we shall be particularly interested in the study of non-Gromov hyperbolicity of the space $(D,k_D)$ where $D$ is a (mostly convex) domain in $\mathbb C^n$. Note that the study of this notion for spaces $(D,k_D)$ where $D$ is convex reduces, because of the affine isomorphism of $D$ to $\mathbb C^k\times D^{\prime}$ with $D^{\prime}$ being convex and Kobayashi hyperbolic, to the study in the situation when $D$ is Kobayashi hyperbolic.

\subsection{Summary of results}
 For a domain $\Omega\subset\mathbb R^n$ denote by $T_{\Omega}$ \textit{the tube domain with the base $\Omega$} as follows
\begin{equation}
T_{\Omega}:=\{z\in\mathbb C^n: \re z\in \Omega\}.
\end{equation}
It is well-known that $T_{\Omega}$ is pseudoconvex iff $\Omega$ is convex (Bochner theorem). Additionally, if $\Omega$ is convex then $T_{\Omega}$ is Kobayashi complete (i. e. $(T_{\Omega},k_{T_{\Omega}})$ is a complete metric space) iff $T_{\Omega}$ is taut iff $T_{\Omega}$ is Kobayashi hyperbolic iff $\Omega$ contains no real line iff $T_{\Omega}$ contains no complex line.

Recall that recent results of A. Zimmer give an almost complete characterization of Gromov hyperbolic spaces $(D,k_D)$ where $D$ is a convex domain (see \cite{Zim 2016a}). Zimmer's paper was a continuation of the study of (non)-Gromov hyperbolicity of domains equipped with the Kobayashi distance presented in a series of papers (see e. g. \cite{Bal-Bon 2000}, \cite{Gau-Ses 2013}, \cite{Nik-Tho-Try 2016}).  In our paper the problem of characterization of the spaces $(T_{\Omega},k_{T_{\Omega}})$, where $\Omega\subset\mathbb R^n$ is a convex domain,  is studied. It is interesting that the results we obtain (see e. g. Corollary~\ref{corollary:gromov-necessary}) show similar geometry of the base $\Omega$ of non-Gromov hyperbolic convex tube domain $(T_{\Omega},k_{T_{\Omega}})$  as the geometry of a non-Gromov hyperbolic convex domain when equipped with the Hilbert metric convex domain in $\mathbb C^n$ (see Proposition 4.1 and Theorem 4.2 in \cite{Kar-Nos 2002}).

In order to obtain the above mentioned result we found a criterion on non-Gromov hyperbolicity of a domain $\Omega\subset\mathbb R^n$ equipped with the pseudodistance satisfying some regularity properties (that are trivially satisfied by the Kobayashi pseudodistance) -- see Theorem~\ref{theorem:criterion}. This allows to conclude the non-Gromov hyperbolicity from the same property of some convex cone associated with the non-smooth boundary point.

We succeeded in getting some sufficient conditions for the non-Gromov hyperbolicity of tube domains with convex bases; however, we failed to get sufficient conditions for Gromov hyperbolicity in the same class of domains. A good way of getting results of that type could rely on a detailed understanding of a behavior of geodesics (both real and complex) in tube domains with smooth bounded convex bases -- in fact a similar idea is used in \cite{Zim 2016a} and \cite{Zim 2016b} when studying the same problem for general ($\mathbb C$-)convex domains. We find a fairly complete description of the continuous extension of complex geodesics up to the boundary in convex tube domains (see Theorem~\ref{theorem:geodesic-continuity}). We heavily rely on methods recently developed in \cite{Zaj 2015a} and \cite{Zaj 2015b}.

In Section~\ref{subsection:geometry-convex} we show an inequality which connects the metric geometry (the Hilbert metric) of the convex base with the Kobayashi distance of the tube domains.

In the last section we present some results loosely related to the following open problem. Is any Kobayashi hyperbolic convex domain biholomorphic to a bounded convex domain?

 \section{Non-Gromov hyperbolicity in tube domains} As already mentioned the starting point for our considerations was the need to understand Gromov hyperbolicity of convex tube domains equipped with the Kobayashi distance.
 In the section below we present sufficient conditions for non-Gromov hyperbolicity of the Kobayashi distance in convex tube domains. But the results are presented in a much more general setting. This is done because the proofs will require only some basic properties of the Kobayashi distance. It is possible that the ideas we present may also find applications in situations other than the ones presented in the paper.

 What we want to present is a criterion on non-Gromov hyperbolicity of metric spaces defined on a class of domains in $\mathbb R^n$  admitting some regularity properties. It will turn out that under these assumptions the non-Gromov hyperbolicity of the metric space defined on a cone will imply the non-Gromov hyperbolicity of the original metric space. The presentation is given for domains in $\mathbb R^n$ equipped with pseudodistances satisfying some straightforward and natural invariance properties (we call such a pair a \textit{ pseudometric space}). Then we apply the results in the case of the Kobayashi pseudodistance, mostly for tube domains. The idea we present relies on a kind of blow-up of a domain near a non-smooth boundary point (typically for convex domains but the condition is formulated for more general ones with some regularity properties of the domain imposed). The blow-up will generate a cone related to the domain whose non-Gromov hyperbolicity will imply the non-Gromov hyperbolicity of the original domain.

\bigskip

Consider a domain $D\subset\mathbb R^n$ and $x\in\partial D$. We say that the pair $(D,x)$ satisfies the property (*) if
for any $v\in\mathbb R^n$ one of the following two properties is satisfied:

\begin{equation}\label{equation:property-one}
(x+[0,\infty)v)\cap D=\emptyset
\end{equation}
or
\begin{equation}\label{equation:property-two}
\text{there is an $\epsilon>0$ such that $x+(0,\epsilon)w\subset D$ for $w\in\mathbb R^n$ sufficiently close to $v$}.
\end{equation}

In such a case we define
\begin{equation}
C_D(x):=\{v\in\mathbb R^n:\text{ the property (\ref{equation:property-two}) is satisfied}\}.
\end{equation}

Note that $C_D(x)$ is an open and connected cone.

Needles to say that if the domain $D\subset\mathbb R^n$ is convex, $x\in\partial D$, then the pair $(D,x)$ satisfies (*).
Moreover, when $D$ is convex then $C_D(x)$ is also convex. Note that in such a situation the closure of $C_D(x)$ is the solid tangent cone to $\overline{D}$ at $x$ appearing in convex geometry and $\partial C_D(x)$ is actually the tangent cone at $x$ to $\overline{D}$. If the convex domain $D$ admits at $x\in\partial D$ only one supporting hyperplane, then the cone $C_D(x)$ is the half space determined by the supporting hyperplane lying on the same side as the domain.



For technical simplification we assume, for the needs of formulation of the next proposition, that $x=0\in\partial D$, $D\subset\mathbb R^n$ is a domain, and the pair $(D,0)$ satisfies (*).

Crucial for our proof is the following property which is satisfied if the pair $(D,0)$ satisfies (*):

for any sequence of positive numbers $t_k$ converging to $\infty$ monotonically we have the convergence $t_kD\to C_D(0)$ in the sense that for any compact $K\subset C_D(0)$ there is a $k_0$ such that for all $k\geq k_0$ we have $K\subset t_kD\subset C_D(0)$.

We may now formulate the result.

\begin{theorem}\label{theorem:criterion}  Let $D\subset\mathbb R^n$ be a domain, $0\in\partial D$. Assume that $(D,0)$ satisfies (*). Assume additionally that $d_{t_kD}$ (respectively, $d_{C_D(0)}$) are pseudodistances on $t_kD$ (respectively, $C_D(0)$)  that
satisfy the following properties
\begin{itemize}
\item{(continuity property)}
$\lim_{k\to\infty}d_{t_kD}(x,y)=d_{C_D(0)}(x,y)$, $x,y\in C_D(0)$,\\
\item{(invariance property)} $d_{t_kD}(t_kx,t_ky,t_kz,t_kw)=d_D(x,y,z,w)$, for any $k$ and  $x,y,z,w\in D$.
\end{itemize}
Then if the space $(C_D(0),d_{C_D(0)})$ is non-Gromov hyperbolic then so is $(D,d_D)$.
\end{theorem}
The invariance property implies that the equality 
$$S_{t_kD}(t_kx,t_ky,t_kz,t_kw)=S_D(x,y,z,w)$$ 
holds for any $k$ and for any $x,y,z,w\in D$ and thus $S_{t_kD}=S_D$ for and $k$.
\begin{proof}
First note that the continuity property gives that
\begin{equation}
\lim_{k\to\infty}S_{t_kD}(x,y,z,w)=S_{C_D(0)}(x,y,z,w),\; x,y,z,w\in C_D(0).
\end{equation}
Because of the non-Gromov hyperbolicity of $(C_D(0),d_{C_D(0)})$ for any $M\in\mathbb R$ we find $x,y,z,w\in C_D(0)$ such that
$S_{C_D(0)}(x,y,z,w)>M$. The continuity property allows us to find a $k_0$ such that for any $k\geq k_0$ we have that $(x,y,z,w\in t_kD$ and $S_{t_kD}(x,y,z,w)>M$. But the invariance property gives
\begin{equation}
S_D\left(\frac{x}{t_k},\frac{y}{t_k},\frac{z}{t_k},\frac{w}{t_k}\right)=S_{t_kD}(x,y,z,w)>M,
\end{equation}
which implies that $S_D>M$. $M$ was chosen arbitrarily so the result follows.
\end{proof}

The continuity property of the Kobayashi pseudodistance defined on arbitrary domains as well as its invariance under biholomorphic mappings (and thus under translation and dilation) allow us to apply the above theorem for convex domains. In other words we have the following.

 \begin{corollary}
 Assume that the domain $D\subset\mathbb C^n$  is convex, $z\in\partial D$. If $(C_D(z),k_{C_D(z)})$ is non-Gromov hyperbolic, then  $(D,k_D)$ is non-Gromov hyperbolic.
 \end{corollary}
 The above result allows us to deduce the non-Gromov hyperbolicity of convex domains from the non-Gromov hyperbolicity of the appropriate convex cone. In applications as a basic model for non-Gromov hyperbolic domain we shall use the most obvious example of the non-Gromov hyperbolic convex domain (when equipped with the Kobayashi distance), which is the polydisc $\mathbb D^n$, $n\geq 2$, that is biholomorphic to the cone $H^n=T_{(0,\infty)^n}$, $H:=\{\lambda\in\mathbb C: \re\lambda>0\}$. This is actually the crucial observation we use below.

 Since special role is played by some convex cones we introduce the notation
 \begin{equation}
 C(e_1,\ldots,e_n,f_1,\ldots,f_n):= \mathbb Re_1+\ldots+\mathbb R e_n+(0,\infty)f_1+\ldots+(0,\infty)f_n,
 \end{equation}
 where  $\{e_1,\ldots,e_n,f_1,\ldots,f_n\}$ is an $\mathbb R$-basis of $\mathbb R^{2n}=\mathbb C^n$ but such that both
 $\{e_1,\ldots,e_n\}$ and $\{f_1,\ldots,f_n\}$ are also $\mathbb C$-bases of $\mathbb C^n$. Then the cone
 $C(e_1,\ldots,e_n,f_1,\ldots,f_n)$ does not contain any complex affine line (and thus it is Kobayashi hyperbolic) but what is more important is that it is affinely isomorphic with the tube domain $T_{(0,\infty)^n}$. The latter  is trivially biholomorphic with the polydisc $\mathbb D^n$. It is standard that the metric space $(\mathbb D^n,k_{\mathbb D^n})$ is non-Gromov hyperbolic for $n\geq 2$. Altogether, we get the following.

  \begin{corollary}
  Assume that the domain  $D\subset\mathbb C^n$, $n\geq 2$, is convex, $z\in\partial D$ and
  \begin{equation}
  C_{D}(z)=C(e_1,\ldots,e_n,f_1,\ldots,f_n),
  \end{equation}
  where $e_j$'s and $f_j$'s are as above. Then $(D,k_D)$ is non-Gromov hyperbolic.
  \end{corollary}

  This result may be applied in the case of tube domains.
  \begin{corollary} Let $\Omega\subset\mathbb R^n$, $n\geq 2$, be a convex domain, $x\in\partial\Omega$. Assume that $C_{\Omega}(x)=(0,\infty)f_1+\ldots+(0,\infty)f_n$, where $\{f_1,\ldots,f_n\}$ is a vector basis of $\mathbb R^n$. Then the pseudometric space  $(T_{\Omega},k_{T_{\Omega}})$ is non-Gromov hyperbolic.
  \end{corollary}

  Finally, we have a  nice necessary condition for the Gromov hyperbolicity of the Kobayashi (pseudo)distance in two-dimensional tube domains.
  \begin{corollary}\label{corollary:gromov-necessary}
  Let $\Omega\subset\mathbb R^2$ be a convex domain such that $(T_{\Omega},k_{T_{\Omega}})$ is Gromov hyperbolic. Then $\Omega$ is strictly convex and its boundary is $C^1$.
  \end{corollary}

  \begin{proof} If $\Omega$ were not strictly convex then there would be a non-trivial line segment $I\subset\partial \Omega$. But then $\partial T_{\Omega}$ contains a non-trivial analytic disc which contradicts the necessary condition of Gromov hyperbolicity of convex domains (Theorem 1.6 in \cite{Zim 2016a}).

   If the boundary were not $C^1$ then, due to the properties of convex domains, it would not be differentiable at some $x\in\partial\Omega$ and thus there would exist two supporting lines to $\Omega$ passing through $x$ and thus $C_{\Omega}(x)$ would be $(0,\infty)f_1+(0,\infty)f_2$ with linearly independent $f_1,f_2$. It suffices to use the previous corollary to get that $(T_{\Omega},k_{T_{\Omega}})$ is  non-Gromov hyperbolic.
   \end{proof}

   \begin{remark} It would be interesting to see whether the previous corollary could be generalized to higher dimensions.

   It is also interesting to be able to formulate some sufficient conditions for the Gromov hyperbolicity of the Kobayashi distance in tube domains. It is not clear even in dimension two for tube domains with bounded bases. The example of a convex tube domain with the unbounded base $\Omega=\{x\in\mathbb R^2: x_1>x_2^2\}$, which is biholomorphic with the unbounded realization of the unit ball (Siegel domain) in
   $\mathbb C^2$ via the map $z\to \left(z_1-\frac{z_2^2}{2},\frac{z_2}{\sqrt{2}}\right)$ (compare Example 6.14 in \cite{Bra-Gau 2016}), gives a tube domain that is Gromov hyperbolic when endowed with the Kobayashi distance.
    \end{remark}

 \begin{remark}
 In this section some necessary conditions for a Gromov hyperbolicity were given which, when applied to convex domains, gave a wide class of non-Gromov hyperbolic domains. In the two-dimensional case the property which guaranteed the non-Gromov hyperbolicity is surprisingly similar to the general situation of the non-Gromov hyperbolicity of the Hilbert metric (compare Proposition 4.1 and Theorem 4.2 in \cite{Kar-Nos 2002} with our Corollary~\ref{corollary:gromov-necessary}).
 \end{remark}

    As a next example of how to apply the above criterion for non-Gromov hyperbolicity we choose another class of domains with a relatively big class of symmetry -- namely Reinhardt domains. Following the same idea as above we get the following necessary condition of Gromov hyperbolicity.

 \begin{corollary}
 Let $D\subset\mathbb C^2$ be a convex Reinhardt domain with the Minkowski functional $h$. Assume that $h$ is not $C^1$ on $\mathbb C^2\setminus\{0\}$. Then $(D,k_D)$ is non-Gromov hyperbolic.
          \end{corollary}

\section{Geodesics (complex and real) in convex tube domains}
Following the ideas presented in \cite{Zim 2016a} and \cite{Zim 2016b} it is probable that a good tool for providing some sufficient conditions for the Gromov hyperbolicity would be the analysis of properties of real geodesics with respect to the Kobayashi distance. The problem of regularity of complex geodesics is also a very important one and may be used in many problems of complex analysis -- one may look for instance at papers \cite{Zim 2016a}, \cite{Zim 2016b}, \cite{Bha-Zim 2016}, \cite{Bha 2016} to name just a few that have appeared recently. The very recent papers of Zaj\c ac allow us to provide many strong regularity properties of complex geodesics in convex tube domains which also show that real geodesics are in fact much more regular (for tube domains with bounded smooth bases) than in the situation studied by Zimmer. We present these results in this section just after presenting some general properties of geodesics in convex tube domains.

 Below we shall be interested in \textit{well behaved} real geodesics, i. e. the real geodesics such that the limits $\lim_{t\to\pm 1}\gamma(t)$ exist (in the extended sense, i. e. as  elements of $\overline{D}\cup\{\infty\}$) -- the expression `a well behaved geodesic'' is taken from \cite{Zim 2016b}, up to a parametrization of a geodesic to the interval $(-\infty,\infty)$.

 \subsection{Geodesics in convex tube domains with bounded bases}
Below we use the notation from the papers \cite{Zaj 2015a} and \cite{Zaj 2015b} and we present some consequences of the results given there. We mainly concentrate on results for tube domains with bounded bases.

Crucial for the problem of regularity of complex geodesics in tube domains will be the results of Zaj\c ac which state that the mapping $h\in H^1(\DD,\mathbb C^n)$ coming up in descriptions of complex geodesics is of the form $a\lambda^2+b\lambda+\bar a$, $a\in\mathbb C^n$, $b\in\mathbb R^n$. The analogous function considered for general bounded convex  (not the strongly convex ones!) domains (whose existence is a consequence of the Lempert Theorem - see \cite{Lem 1981}, \cite{Roy-Wong 1983} or \cite{Jar-Pfl 1993}) is far from being so regular.

The theorem below follows directly from Theorem 3.3 in \cite{Zaj 2015a}, Theorem 3.1 and Remark 3.3 in \cite{Zaj 2015b}. Let us draw the attention of the Reader to the fact that we restrict ourselves to the cases of tube domains with bounded bases. This allows us, when working with complex geodesics, to restrict to boundary measures as studied by Zaj\c ac being absolutely continuous with respect to the Lebesgue measure.

For the convex domain $\Omega\subset\mathbb R^n$ and $v\in\mathbb R^n$ put
\begin{equation}
P_{\Omega}(v):=\{x\in\overline{\Omega}:\langle y-x,v\rangle<0\text{ for all $y\in\Omega$}\}\subset\partial\Omega.
\end{equation}

\begin{theorem}[see \cite{Zaj 2015a}, \cite{Zaj 2015b}] \label{theorem:zajac-characterization}
Let $\Omega\subset\mathbb R^n$ be a bounded convex domain. Then $f:\mathbb D\to T_{\Omega}$ is a complex geodesic in $T_{\Omega}$  iff there exist $a\in\mathbb C^n$, $b\in\mathbb R^n$, not both equal to $0$, and $g\in L^1(\mathbb T,\mathbb R^n)$ such that
\begin{equation}
f(\lambda)=\frac{1}{2\pi}\int_{\mathbb T}\frac{\xi+\lambda}{\xi-\lambda}g(\xi)d\mathcal L^{\mathbb T}(\xi)+i \im f(0)
\end{equation}
and $g(\lambda)\in P_{\Omega}(\bar\lambda h(\lambda))$ for a. a. $\lambda\in\mathbb T$, where $h(\lambda):=a\lambda^2+b\lambda+\bar a$.
\end{theorem}
In the above formulation note that  $\bar\lambda h(\lambda)\in\mathbb R^n$, $\lambda\in\mathbb T$, thus the expression $P_{\Omega}(\bar\lambda h(\lambda))$ makes sense.

Note that $\re f^*(\lambda)=g(\lambda)$ for a. a. $\lambda\in\mathbb T$.

\subsection{Geodesic in convex tube domains -- general properties}\label{subsection:geodesics-general} Following the ideas of uniqueness of complex geodesics in general convex domains for the case of tube domains we get relatively easily some basic properties.

Assume that we have two distinct points $w,z\in T_{\Omega}$. In case $T_{\Omega}$ is Kobayashi hyperbolic there exists a complex geodesic passing through $w$ and $z$. In many cases it is uniquely determined.

\begin{proposition}\label{proposition:uniqueness} Let $\Omega\subset\mathbb R^n$ be a strictly convex bounded domain.
Let $w,z\in T_{\Omega}$ be two distinct points. Then, up to an automorphism of the unit disc, there is only one complex geodesic passing through $w$ and $z$.
 \end{proposition}
 \begin{proof} Let $f,g:\mathbb D\to T_{\Omega}$ be complex geodesics such that $f(0)=g(0)=w$, $f(s)=g(s)=z$, $s\in(0,1)$. Then (e. g. see Theorem~\ref{theorem:zajac-characterization} and remark after it) the radial limits $\re f^*$ and $\re g^*$ exist and belong to $\partial\Omega$ a. e. on $\mathbb T$. But then the function $\frac{f+g}{2}:\mathbb D\to T_{\Omega}$ is also a geodesic (passing through $w$ and $z$) and the radial limit $\frac{\re (f+g)^*}{2}$ exists and belongs to $\partial \Omega$ a. e. on $\mathbb T$. Then the strict convexity of $\Omega$ implies that $\re f^*=\re g^*$ a. e. on $\mathbb T$. But the real parts of radial limits determine uniquely the complex geodesics (use Theorem~\ref{theorem:zajac-characterization}), so $f\equiv g$ -- a contradiction.
 \end{proof}

\begin{proposition}
Let $\Omega\subset\mathbb R^n$ be a convex domain that contains no real line. Let $w,z\in\Omega=T_{\Omega}\cap\mathbb R^n$ be distinct. Then there is a complex geodesic $f:\mathbb D\to T_{\Omega}$ such that $f(0)=w$, $f(s)=z$, $s\in(0,1)$ and $f((-1,1))\subset\mathbb R^n$. In particular, for any distinct $w,z\in\Omega\subset T_{\Omega}\cap\mathbb R^n$ there is always a real geodesic (for $T_{\Omega}$) passing through $w,z$ lying entirely in $\Omega\subset\mathbb R^n$.
 \end{proposition}
 \begin{proof}
 Let $g:\mathbb D\to T_{\Omega}$ be a complex geodesic such that $g(0)=w$, $g(s)=z$, $s\in(0,1)$. Then the mapping
 $f:\mathbb D\to T_{\Omega}$ defined by the formula $f(\lambda):=\frac{g(\lambda)+\overline{g(\overline{\lambda})}}{2}$, $\lambda\in\mathbb D$, is also a complex geodesic satisfying the desired property.
 \end{proof}

\subsection{Complex geodesics in tube domains over convex, smooth and bounded bases}
In the construction below we present how Theorem~\ref{theorem:zajac-characterization}  implies the continuity of complex (and real) geodesics up to the boundary.

Below we shall assume additionally that $\Omega$ has $C^k$ boundary ($k\geq 2$) and is strictly convex. Note that in such a case we have the following properties.
For any $v\in\mathbb C^n\setminus\{0\}$ the set $P_{\Omega}(v)$ contains exactly one point. Moreover, for $x\in\partial\Omega$ we get that $P_{\Omega}(\nu_{\Omega}(x))=\{x\}$, where $\nu_{\Omega}(x)=\frac{\nabla\rho(x)}{||\nabla\rho(x)||}$, where $\rho$ is the defining function of $\Omega$ near $x$, denotes the unit outer normal vector to $\partial \Omega$ at $x$. This gives us the following mapping

\begin{equation}
\Phi:\partial \Omega\owns x\to\nu_{\Omega}(x)\in \mathbb S^{n-1}\subset\mathbb R^n,
\end{equation}
which is $C^{k-1}$-smooth, injective and onto (here we need the strict convexity of $\Omega$, its boundedness and smoothness!). Consequently, $\Phi$ is a $C^{k-1}$-diffeomorphism.

In such a situation we see that $P_{\Omega}(w)=\Phi^{-1}(\{w\})$, $w\in\mathbb S^{n-1}$.

Therefore, in the situation as in Theorem \ref{theorem:zajac-characterization}, we have
\begin{equation}\label{equation:function-P}
P_{\Omega}(\bar\lambda h(\lambda))=\Phi^{-1}\left(\frac{\bar\lambda h(\lambda)}{||\bar\lambda h(\lambda)||}\right)
\end{equation}
for a. a. $\lambda\in\mathbb T$.

Note that $\bar\lambda h(\lambda)=2\re(a\lambda)+b$, $\lambda\in\mathbb T$. Therefore,
the expression on the right side of (\ref{equation:function-P}) is well defined for all but at most two points $\lambda\in\mathbb T$.

Consequently, we have

\begin{equation}
\re f^*(\lambda)=g(\lambda)=\Phi^{-1}\left(\frac{\bar\lambda h(\lambda)}{||\bar\lambda h(\lambda)||}\right)=\Phi^{-1}\left(\frac{2\re(a\lambda)+b}{||2\re(a\lambda)+b||}\right)
\end{equation}
for a. a. $\lambda\in\mathbb T$.

Therefore, for the form of geodesics (and their regularity) the behaviour of the projection onto $\mathbb S^{n-1}$ of the following mapping is crucial
\begin{equation}
\tilde F:\mathbb T\owns\lambda\to 2\re(a\lambda)+b\in\mathbb R^n.
\end{equation}
Recall that the above mapping (so the chosen $a,b$) have to be such that it is not identically equal to $0$. Moreover, the assumption $\re f(0)\in\Omega$ implies that the projection of the image of the above mapping onto $\mathbb S^{n-1}$ is not a singleton.

The detailed study of the form of the above mapping gives the following possibilities for the mapping (defined for all but at most two elements of $\mathbb T$) $F:=\frac{\tilde F}{||\tilde F||}$,
\begin{equation}
F:\mathbb T\owns\lambda\to\frac{\bar\lambda h(\lambda)}{||\bar\lambda h(\lambda)||}\in\mathbb S^{n-1}.
\end{equation}
The situations which we list below reflect the fact that the image of the mapping $\tilde F:\mathbb T\owns\to2\re(a\lambda)+b$ has two possibilities with further subcases. First note that the image of the mapping is either an ellipse (coplanar with the origin or not), which is the case when the vectors $\re a$ and $\im a$ are $\mathbb R$--linearly independent or a closed line segment (with the origin lying in the segment or not).

Note that the case when $\tilde F$ has the image being the line segment with $0$ lying in its boundary is impossible because then the mapping $f$ would be constant. Below we present all the possibilities we have to study
\begin{itemize}
\item $F$ is a linear embedding of the circle into a circle,\\
\item $F$ is a real analytic mapping with the image being the (closed) smaller arc of a great circle,\\
\item $F:\mathbb T\setminus\{\lambda_0\}\to\mathbb S^{n-1}$ is a real analytic diffeomorphism onto the image being the big open semicircle such that
\begin{equation}
\lim_{t\to t_0^+}F(e^{it})=-\lim_{t\to t_0^-}F(e^{it}),\; \lambda_0=e^{it_0},
\end{equation}
\\
\item $F:\mathbb T\setminus\{\lambda_0,\lambda_1\}\to\mathbb S^{n-1}$ is constant on the two connected components (arcs) of $\mathbb T\setminus\{\lambda_0,\lambda_1\}$ and the two values are opposite.
\end{itemize}
Let us underline here that it also follows from the definition (by the appropriate choice of $a,b$) that all the possibilities listed above do occur.

Note also that the singular points ($\lambda_0,\lambda_1$ from the above description) are the ones such that $\tilde F(\lambda)=0$).

Now the composition of $F$ with the diffeomorphism $\Phi^{-1}$ gives the radial limit function $\re f^*$, which preserves the $C^{k-1}$-smoothness on the boundary (with the exception of $\lambda_0,\lambda_1$). In the theorem below we present a continuity result of complex (and consequently also of real) geodesics in $T_{\Omega}$ up to the boundary, which we formulate only for $C^2$-smooth domains. A more general result with some smoothness up to the boundary (depending on $k$) is also possible but we restrict ourselves to continuity results for which $C^2$-smoothness is sufficient. And although it could be conceivable that results on continuous extension of geodesics could be obtained with simpler tools we presented a more general attitude so that it could be used in other, potentially more refined, applications.

\begin{theorem}\label{theorem:geodesic-continuity} Let $\Omega\subset\mathbb R^n$ be a strictly convex, bounded $C^2$-smooth domain. Let $f:\mathbb D\to T_{\Omega}$ be a complex geodesic. Then only the following may happen
\begin{itemize}
\item $f$ extends to a continuous mapping on $\bar{\mathbb D}$ with $(\re f)(\mathbb T)\subset\partial\Omega$,\\
\item there is a $\lambda_0=e^{it_0}\in\mathbb T$ such that $f$ extends to a continuous mapping on $\overline{\mathbb D}\setminus\{\lambda_0\}$ (denoted again by the same symbol $f$), the limits $x_+:=\lim_{t\to t_0^+}\re f(e^{it})$, $x_-:=\lim_{t\to t_0^-}\re f(e^{it})$ exist, $x_+\neq x_-$, and they are both from $\partial\Omega$. Moreover, the following limits exist and satisfy the additional properties:  $\lim_{r\to 1^-}\re f(re^{it_0})\in[x_-,x_+]$,
$\lim_{r\to 1^-}\im f(re^{it_0})$ equals $\infty$ or $-\infty$,\\
\item there are distinct points $\lambda_0,\lambda_1\in\mathbb T$ such that $f$ extends to a continuous mapping on $\overline{\mathbb D}\setminus\{\lambda_0,\lambda_1\}$, $\re f$ attains two different values $x_0,x_1\in\partial\Omega$ on $\mathbb T\setminus\{\lambda_0,\lambda_1\}$,
\begin{equation}
\lim_{r\to 1^-}\re f(\gamma(r)),\lim_{r\to 1^+}\re f(\gamma(r))\in[x_0,x_1]
\end{equation}
 and limits $\lim_{r\to \pm1}\im f(\gamma(r))$ exist and one of them is $\infty$ and the other one $-\infty$.\\
In the formula above $\gamma(r)=a(r)$, $r\in(-1,1)$, where $a$ is an automorphism of $\mathbb D$ such that $a(-1)=e^{it_0}=\lambda_0$, $a(1)=e^{it_1}=\lambda_1$.
\end{itemize}
In particular, all the real geodesics in the domain $T_{\Omega}$ behave well and their limits (at $\pm 1$ are different).

Additionally, all the possibilities listed above do occur.
\end{theorem}
\begin{proof}
Note that $\re f^*=\Phi^{-1}\circ F$ a. e. on $\mathbb T$. The latter is $C^1$ on $\mathbb T$ (with the exception of at most two points $\lambda_0,\lambda_1$).

The continuity of the mapping $f$ on $\overline{\mathbb D}$ without at most two singular points follows from the fact that $\re f^*$
is $C^1$ (it would be sufficient if it were Dini continuous).  This follows form the general theory of conjugation operators on harmonic functions (the standard reference is \cite{Dur 1970} or \cite{Gar 2007}).

The existence of limits follows from the standard reasoning on the boundary regularity properties of the functions defined by the Poisson kernel.

As to the existence of limits of real geodesics there is nothing to do when there is no singularity in $\mathbb T$ or there are two of them (then the function may be calculated explicitly). Therefore, it is sufficient to consider the existence of the following limits (both real and imaginary part)
\begin{equation}
\lim_{r\to 1^-}\frac{1}{2\pi}\int_{\mathbb T}\frac{e^{it}+r}{e^{it}-r}u(e^{it})d\mathcal L^{\mathbb T}(e^{it}),
\end{equation}
where $u:\mathbb T\setminus\{1\}\to\mathbb R$ is continuous and $0<u^{-}(1)<u^{+}(1)$, where $u^{\pm}(1):=\lim_{t\to 0^{\pm}}u(e^{it})$. It follows from the fact that the real part of the expression above gives the value of the solution of the Dirichlet problem with the boundary data $u$ that the limit of the real part exists and is a number lying in the interval $(u_-(1),u^+(1))$. As to the imaginary part note that we may assume that $u(t)>3\delta>0$ whereas $0<u(-t)<\delta$ for $t\in (0,t_0)$ and the imaginary part of the integrand is $\frac{-2\sin t u(e^{it})}{\sqrt{1+r^2-2r\cos t}}$. Simple analysis gives the desired limit equal to $-\infty$.
\end{proof}
\begin{remark}
The above theorem gives a precise description of the regularity of complex geodesics. It also gives the proof of the well behavior of all real geodesics in domains considered in the theorem. Recall that in \cite{Zim 2016b} such a phenomenon is shown for locally $m$-convex domains (Corollary 7.10 there). Our result applies to more general domains (though restricted to the special case of tube domains). It also gives much more information as to the continuity property of the complex geodesics is concerned - note that the tube domains (even with bounded bases) are neither bounded nor smooth.
\end{remark}

Let us close this subsection with a result on the existence of complex geodesics passing through points from the boundary (compare similar results on strictly linearly convex domains, e. g. \cite{Cha-Hu-Lee 1988}).

\begin{proposition} Let $\Omega\subset\mathbb R^n$ be as in the previous theorem. Let $x,y\in\partial\Omega$ be distinct. Then there is a complex geodesic $f:\mathbb D\to T_{\Omega}$ that extends continuously through $\pm 1$ and such that $f(-1)=x$, $f(1)=y$.
\end{proposition}
\begin{proof}
It follows from the description of geodesics in the previous subsection that such a geodesic would exist if for any distinct $u,v\in \mathbb S^{n-1}\subset\mathbb R^n$ there would exist $a\in\mathbb C^n$, $b\in\mathbb R^n$ such that $F(-1)=u$, $F(1)=v$, where $F(\lambda)=\frac{2\re (a\lambda)+b}{||2\re(a\lambda)+b||}$. Since the mapping
\begin{equation}
\mathbb R^n\times\mathbb R^n\owns(a,b)\to(2a+b,-2a+b)\in\mathbb R^n\times\mathbb R^n
\end{equation}
is an isomorphism we easily find $a,b$ as required (we may even take $a\in\mathbb R^n$).
\end{proof}

\subsection{Geometry of convex domains}\label{subsection:geometry-convex}
As we saw in Subsection~\ref{subsection:geodesics-general} there is a sense in  considering the restriction of the Kobayashi pseudodistance in tube domains to the base. In other words it could potentially be interesting to understand what the properties ${k_{T_{\Omega}}}_{|\Omega\times\Omega}$ has and what its relations with other naturally equipped metrics are. We present some properties which also may have future applications.

First recall the notion of locally $m$-convex set. It is defined in the complex setting. But it may also be defined in the real setting.

For domain $D\subset\mathbb K^n$ ($\mathbb K=\mathbb C$ or $\mathbb R$), $p\in D$, non-zero $v\in\mathbb K^n$ we define \begin{multline}
\delta_D(p):=\inf\{||x-p||:x\in\mathbb K^n\setminus D\},\\
\delta_D(p;v):=\inf\{||x-p||:x\in(p+\mathbb K v)\cap(\mathbb K^n\setminus D)\}.
\end{multline}
We call a proper convex domain $D\subset\mathbb K$ \textit{locally $m$-convex} if for every $R>0$ there is a $C>0$ such that $\delta_D(p;v)\leq C\delta_D(p)^{1/m}$ for all $p\in D$, $||p||<R$ and non-zero $v\in\mathbb K^n$.

  We have the following property for the convex domain $\Omega\subset\mathbb R^n$:

$T_{\Omega}$ is locally $m$-convex iff $\Omega$ is locally $m$-convex.

This shows that the continuity of real geodesics in (some) convex tube domains may also be deduced from results in \cite{Zim 2016a}.

Recall that for a convex bounded domain $\Omega\subset\mathbb R^n$ one may define \textit{the Hilbert metric} as follows. Let $x,y\in\Omega$ be distinct and let $\alpha,\beta\in\partial\Omega$ be points from $\partial\Omega$ which are points of intersections of the line passing through $x$ and $y$. Let $\alpha$ be the point closer to $x$ and $\beta$ the one closer to $y$. Define
\begin{equation}
h_{\Omega}(x,y):=\log\frac{||x-\alpha||\cdot ||y-\beta||}{||x-\beta||\cdot ||y-\alpha||}.
\end{equation}
Additionally we put $h_{\Omega}(x,x):=0$, $x\in\Omega$. Then $(\Omega,h_{\Omega})$ is a metric space. We have the following relation of $h_{\Omega}$ with the Kobayashi distance $k_{T_{\Omega}}$

\begin{proposition} Let $\Omega\subset\mathbb R^n$ be a bounded convex domain. Then
\begin{equation}
h_{\Omega}(x,y)\geq2 k_{T_{\Omega}}(x,y),\;(x,y\in\Omega\subset T_{\Omega}\cap\mathbb R^n.
\end{equation}
\end{proposition}
\begin{proof}
Let $x,y\in\Omega$ be distinct. Let $\alpha,\beta$ be as in the definition of the Hilbert metric. Without loss of generality $\beta=-\alpha$. Put $x=t\alpha$, $y=s\alpha$. Then $-1<s<t<1$. We then have
\begin{equation}
h_{\Omega}(x,y)=\log\frac{(1-t)(1+s)}{(1+t)(1-s)}=2k_{\mathbb D}(s,t)\geq 2k_{T_{\Omega}}(s\alpha,t\alpha)=2k_{T_{\Omega}}(x,y).
\end{equation}
\end{proof}

\section{Is the Kobayashi hyperbolic convex domain biholomorphic to a bounded convex domain?}
A convex domain is linearly isomorphic with the product of $\mathbb C^k$ and some Kobayashi hyperbolic convex domain. Additionally, a convex Kobayashi hyperbolic domain is biholomorphic to a bounded domain. But whether one may take as the last domain a convex one is not known (see \cite{For-Kim 2015}). With any convex domain $D\subset\mathbb R^n$ we may relate the convex cone $S(D):=\{v\in\mathbb R^n:a+[0,\infty)v\subset D\}$ with some (equivalently, any) $a\in D$ (see e. g. \cite{Zwo 1999}, \cite{Zaj 2015b}).

The Kobayashi hyperbolicity of the convex domain $D\subset\mathbb C^n$ is in this language equivalent to the fact that $S(D)$ contains no complex line. It is however possible that $S(D)$ contains an $n$-dimensional real space (as a subspace of $\mathbb R^{2n}=\mathbb C^n$) and no complex line. In such a case we shall deal (up to a affine complex isomorphism) with the Kobayashi hyperbolic tube domain.

One may try to study the problem mentioned above by considering the more and more complex structure of $S(D)$. The case $S(D)=\{0\}$ is equivalent to the fact that $D$ is bounded. The second simplest case ($S(D)$ being a half line) seems to be already a problem. In such a case (even a little more general) we may however find a weaker answer. Namely, the domain is then biholomorphic to a bounded $\mathbb C$-convex domain. More precisely, we have the following proposition whose proof actually follows the idea presented in \cite{Zim 2016b}.

\begin{proposition} Let $D$ be a hyperbolic convex domain in $\mathbb C^n$ such that $S(D)$  is properly contained in real two-dimensional half space. Then $D$ is biholomorphic with the bounded $\mathbb C$-convex domain.
\end{proposition}
Recall that a domain $D\subset\mathbb C^n$ is called \textit{$\mathbb C$-convex} if for any complex affine line $l$ intersecting $D$ the set $D\cap l$ is connected and simply connected.
\begin{proof}
Without loss of generality we may assume that
\begin{equation}
D\subset\{(z^{\prime},z_n)\in\mathbb C^n:\im z_n>f(z^{\prime})\},
\end{equation}
where $f:\mathbb C^{n-1}\to[0,\infty)$ is convex and such that $\liminf_{||z^{\prime}||\to\infty}\frac{f(z^{\prime})}{||z^{\prime}||}>0$. Then the map
\begin{equation}
z\to\left(\frac{1}{i+z_n},\frac{z^{\prime}}{i+z_n}\right)
\end{equation}
preserves the $\mathbb C$-convexity and the image of $D$ is bounded.
\end{proof}

Actually, the result of Zimmer (e. g. Proposition 1.9 in \cite{Zim 2016b}) suggests that the set defined there
\begin{equation}
\{(z^{\prime},z_n)\in\mathbb C^n:\im\left(\frac{1}{z_n}\right)|z_n|>||z^{\prime}||_p\},\;p>1
\end{equation}
is a possible candidate for the convex domain which is not biholomorphic to a bounded convex domain. As proven in \cite{Zim 2016b} this set is biholomorphic to a $\mathbb C$-convex bounded domain with a non-trivial analytic disc in the boundary. Its possible convex and bounded realization could not have this property; because as proven by Zimmer the set when endowed with the Kobayashi distance must be Gromov hyperbolic.

\subsection{Some examples}
 \begin{example} As another candidate for a counterexample could serve, for instance, the tube domain with the basis
 \begin{equation}
 \Omega:=\{x\in\mathbb R^2:x_1+x_2>1,x_1,x_2>0\}.
 \end{equation}
 \end{example}

 On the other hand an example studied in \cite{Zaj 2015b} (Example 3.7) is, as we see it below, biholomorphic to a bounded convex domain.

 \begin{example} Let us consider the tube domain with the base
 \begin{equation}
 \Omega:=\{x\in(0,\infty)^2:x_1x_2>1\}.
 \end{equation}
 The mapping
 \begin{equation}
 \Phi:z\to\left(\frac{1-z_1}{1+z_2},\frac{1-z_2}{1+z_2}\right)
 \end{equation}
 is an involution and $\Phi(T_{\Omega})$ lies in $\mathbb D^2$ and it may be described as
 \begin{equation}
 (1-|z_1|^2)(1-|z_2|^2)>|1+z_1|^2|1+z_2|^2,\; z_1,z_2\in\mathbb D.
 \end{equation}
It is elementary to see that the function $\rho(z):=\log|1+z_1|^2+\log|1+z_2|^2-\log(1-|z_1|^2)-\log(1-|z_2|^2)$
defines the set in $\mathbb D^2$, i. e. its gradient does not vanish on $\mathbb D^2$ and $\mathbb D^2\cap\Phi(T_{\Omega})=\{z\in\mathbb D^2:\rho(z)<0\}$. Moreover, the points $\left(\partial \Phi(T_{\Omega})\right)\cap\mathbb D^2$ are points of strong pseudoconvexity (but not of the strict convexity). Even more one has the equality
\begin{equation}
\overline{\Phi(T_{\Omega})}\cap\partial \mathbb D^2=\left(\{-1\}\times\overline{\mathbb D}\right)\cup\left(\overline{\mathbb D}\times\{-1\}\right).
\end{equation}
On the other hand (we put $\triangle(\lambda_0,r):=\{\lambda\in\mathbb C:|\lambda-\lambda_0|<r\}$, $\lambda_0\in\mathbb C$, $r>0$)
\begin{multline}
\partial \Phi(T_{\Omega})\cap\mathbb D^2=\Phi(\{x_1x_2=1,x_1,x_2>0\})=\\
\left\{\left(\frac{1-z_1}{1+z_2},\frac{1-z_2}{1+z_2}\right):\re z_1\re z_2=1, \re z_1,\re z_2>0\right\}=\\
\bigcup_{x\in(-1,1)}\partial\triangle\left(\frac{x+1}{2},\frac{1-x}{2}\right)\times\partial\triangle\left(\frac{-x+1}{2},\frac{1+x}{2}\right).
\end{multline}
And consequently
\begin{equation}
\Phi(T_{\Omega})=\bigcup_{x\in(-1,1)}\triangle\left(\frac{x+1}{2},\frac{1-x}{2}\right)\times\triangle\left(\frac{-x+1}{2},\frac{1+x}{2}\right).
\end{equation}
Therefore, the set $\Phi(T_{\Omega})$ is convex. It is however not strictly convex, for instance
\begin{equation}
\left\{(x,-x):x\in[-1,1]\right\}\subset\partial\Phi(T_{\Omega}).
\end{equation}

  \end{example}

\begin{remark} Using the results by S.Shimizu (see \cite{Shi 2000}) it is known that any Kobayashi hyperbolic convex tube domain $T_\Omega$ in $\mathbb C^2$ is biholomorphically equivalent to one of the following tubes $T_{\Omega_j}$, $j=1,..,4$:
\begin{itemize}
\item $\Omega_1=\{y\in\mathbb R^2: y_2>y_1^2\}$,
\item $\Omega_2=\{y\in\mathbb R^2:y_1>0, y_2>0\}$,
\item $\Omega_3=\{y\in\mathbb R^2:y_2>e^{y_1}\}$,
\item $T_{\Omega_4}$ with $\Aut(T_{\Omega_4})\leq 3$.
  \end{itemize} 
Recall that $T_{\Omega_1}$ is biholomorphic to a ball and $T_{\Omega_2}$ to a bidisc. It remains to see which of the other tube domains are biholomorphically equivalent to a bounded convex domain. It seems possible to answer this question for the third case and the fourth one when $\dim\Aut(T_{\Omega_4})=3$ using the characterization of (2,3)-manifolds from \cite{Isa 2007}. We will try to come back to this problem.
\end{remark}


\begin{thebibliography}{10}

\bibitem{Bal-Bon 2000} \textsc{Z. M. Balogh, M. Bonk}, \textit{Gromov hyperbolicity and the Kobayashi metric on
strictly pseudoconvex domains}, Comment. Math. Helv., 75(3), 504--533, (2000).

\bibitem{Bha 2016} \textsc{G. Bharali}, \textit{Complex geodesics, their boundary regularity, and a Hardy--Littlewood--type lemma},
Ann. Acad. Sci. Fennicae. Math., to appear, (2016).

\bibitem{Bha-Zim 2016} \textsc{G. Bharali, A. Zimmer}, \textit{Goldilocks domains, a weak notion of visibility, and applications}, preprint arXiv:1602.01742, (2016).

\bibitem{Bra-Gau 2016} \textsc{F. Bracci, H. Gaussier}, \textit{Horosphere topology}, preprint arXiv:1605.04119, (2016).

\bibitem{Bra-Sar 2009} \textsc{F. Bracci, A. Saracco}, \textit{Hyperbolicity in unbounded convex domains}, Forum Mathematicum, 5, 815--826, (2009).

\bibitem{Cha-Hu-Lee 1988} \textsc{Chin-Huei Chang, M. C. Hu, Hsuan-Pei Lee}, \textit{Extremal Analytic Discs With Prescribed Boundary Data}, Trans. Amer. Math. Soc., 310, (1988), 355--369.

\bibitem{Dur 1970} \textsc{P. J. Duren}, Theory of $H^p$ spaces, New York ; London : Academic Press, (1970).

\bibitem{For-Kim 2015} \textsc{J. E. Fornaess, K.-T. Kim}, \textit{ Some Problems},
Complex Analysis and Geometry, 369--377, Springer Proc. Math. Stat., 144, Springer, Tokyo, 2015.

\bibitem{Gar 2007} \textsc{J. B. Garnett}, \textit{Bounded Analytic Functions}, Graduate Texts in Mathematics, Springer, (2007).

\bibitem{Gau-Ses 2013} \textsc{H. Gaussier, H. Seshadri}, \textit{On the Gromov hyperbolicity of convex domains in $\mathbb C^n$}, arxiv.org/abs/1312.0368.

\bibitem{Isa 2007} \textsc{A. Isaev}, \textit{Lectures on the Automorphism Groups of Kobayashi-Hyperbolic Manifolds}, Springer Lecture Notes in Mathematics 1902 (2007).

\bibitem{Jar-Pfl 1993} \textsc{M. Jarnicki, P. Pflug}, \textit{Invariant Distances and Metrics in Complex Analysis-Second Edition}, de Gruyter Expositions in Mathematics 9, Walter de Gruyter, 2013.

\bibitem{Kar-Nos 2002} \textsc{A. Karlsson, G. A. Noskov}, \textit{The Hilbert metric and Gromov hyperbolicity}, Enseign. Math. (2), 48, (2002), 73--89.

\bibitem{Lem 1981} \textsc{L. Lempert}, \textit{ La m\'etrique de Kobayashi et la repr\'esentation des domaines sur la boule}, Bull. Soc. Math. France, 109, (1981), 427--474.

\bibitem{Nik-Tho-Try 2016} \textsc{N. Nikolov, P. J. Thomas, M. Trybu\l a}, \textit{Gromov (non)hyperbolicity of certain domains
in $\mathbb C^2$}, Forum Math., 28(4), (2016), 783--794.

\bibitem{Shi 2000} \textsc{S. Shimizu}, \textit{Classification of two-dimensional tube domains}, Amer.Journal of Math. 122, (2000), 289--1308.

\bibitem{Roy-Wong 1983} \textsc{H. L. Royden, P.-M. Wong}, \textit{Carath\'eodory and Kobayashi metric on convex domains}, preprint, (1983).

\bibitem{Zaj 2015a} \textsc{S. Zaj\c ac}, \textit{Complex geodesics in convex tube domains}, Ann. Scuola Norm-Sci., (2015), 1337-1361.

\bibitem{Zaj 2015b} \textsc{S. Zaj\c ac}, \textit{Complex geodesics in convex tube domains II}, Ann. Mat. Pura Appl., (2015), 1--23.

\bibitem{Zim 2016a} \textsc{A. M. Zimmer}, \textit{Gromov hyperbolicity and the Kobayashi metric on convex domains of finite type}, Math. Ann., 365, (2016), 1425--1498.

\bibitem{Zim 2016b} \textsc{A.  M. Zimmer}, \textit{Gromov hyperbolicity, the Kobayashi metric, and $\mathbb C$-convex sets}, Trans. Amer. Math. Soc., to appear, (2016).

\bibitem{Zwo 1999} \textsc{W. Zwonek}, \text{On hyperbolicity of pseudoconvex Reinhardt domains}, Archiv Math.  (Basel), 72, (1999), 304--314.

\end{thebibliography}
\end{document}